\documentclass{article}
\usepackage{graphicx} 
\usepackage{graphicx}
\usepackage{amsmath,amsthm,amssymb,enumerate}
\usepackage{euscript,mathrsfs}
\usepackage{color}
\usepackage{dsfont}
\usepackage[left=2cm,right=2cm,top=3.5cm,bottom=3.5cm]{geometry}
\usepackage{color}
\usepackage[framemethod=tikz]{mdframed}
\allowdisplaybreaks

\usepackage{soul}
\usepackage{esint}
\catcode`\@=11 \@addtoreset{equation}{section}

\catcode`\@=12

\allowdisplaybreaks

\newtheorem{Theorem}{Theorem}[section]
\newtheorem{Proposition}[Theorem]{Proposition}
\newtheorem{Lemma}[Theorem]{Lemma}
\newtheorem{Corollary}[Theorem]{Corollary}

\theoremstyle{definition}
\newtheorem{Definition}[Theorem]{Definition}

\newtheorem{Remark}[Theorem]{Remark}

\newcommand{\bTheorem}[1]{
\begin{Theorem} \label{T#1} }
\newcommand{\eT}{\end{Theorem}}

\newcommand{\bProposition}[1]{
\begin{Proposition} \label{P#1}}
\newcommand{\eP}{\end{Proposition}}

\newcommand{\bLemma}[1]{
\begin{Lemma} \label{L#1} }
\newcommand{\eL}{\end{Lemma}}

\newcommand{\bCorollary}[1]{
\begin{Corollary} \label{C#1} }
\newcommand{\eC}{\end{Corollary}}

\newcommand{\bRemark}[1]{
\begin{Remark} \label{R#1} }
\newcommand{\eR}{\end{Remark}}

\newcommand{\bDefinition}[1]{
\begin{Definition} \label{D#1} }
\newcommand{\eD}{\end{Definition}}

\newcommand{\bfphi}{\boldsymbol{\varphi}}

\newcommand{\bFormula}[1]{
\begin{equation} \label{#1}}
\newcommand{\eF}{\end{equation}}

\newcommand{\Ov}[1]{\overline{#1}}

\newcommand{\vr}{\varrho}

\newcommand{\vt}{\vartheta}

\newcommand{\vm}{\vc{m}}

\newcommand{\vn}{\vc{n}}

\newcommand{\vc}[1]{{\bf #1}}

\newcommand{\Div}{{\rm div}_x}
\newcommand{\Grad}{\nabla_x}

\newcommand{\dx}{\,{\rm d} {x}}

\newcommand{\dt}{\,{\rm d} t }

\newcommand{\intO}[1]{\int_{\Omega} #1 \ \dx}

\newcommand{\vv}{\vc{v}}

\newcommand{\D}{{\rm d}}

\newcommand{\ep}{\varepsilon}

\newcommand{\R}{\mathbb{R}}

\newcommand{\br}{ \nonumber \\ }

\def\softd{{\leavevmode\setbox1=\hbox{d}%
          \hbox to 1.05\wd1{d\kern-0.4ex{\char039}\hss}}}
\definecolor{Cgrey}{rgb}{0.85,0.85,0.85}
\definecolor{Cblue}{rgb}{0.50,0.85,0.85}
\definecolor{Cred}{rgb}{1,0,0}
\definecolor{fancy}{rgb}{0.10,0.85,0.10}

\newcommand\Cbox[2]{%
    \newbox\contentbox%
    \newbox\bkgdbox%
    \setbox\contentbox\hbox to \hsize{%
        \vtop{
            \kern\columnsep
            \hbox to \hsize{%
                \kern\columnsep%
                \advance\hsize by -2\columnsep%
                \setlength{\textwidth}{\hsize}%
                \vbox{
                    \parskip=\baselineskip
                    \parindent=0bp
                    #2
                }%
                \kern\columnsep%
            }%
            \kern\columnsep%
        }%
    }%
    \setbox\bkgdbox\vbox{
        \color{#1}
        \hrule width  \wd\contentbox %
               height \ht\contentbox %
               depth  \dp\contentbox
        \color{black}
    }%
    \wd\bkgdbox=0bp%
    \vbox{\hbox to \hsize{\box\bkgdbox\box\contentbox}}%
    \vskip\baselineskip%
}

\mdfdefinestyle{MyFrame}{%
	linecolor=black,
	outerlinewidth=1pt,
	roundcorner=5pt,
	innertopmargin=\baselineskip,
	innerbottommargin=\baselineskip,
	innerrightmargin=10pt,
	innerleftmargin=10pt,
	backgroundcolor=white!20!white}

\date{}

\begin{document}


\title{Maximal entropy production principle and\\the Euler system of gas dynamics}

\author{Elisabetta Chiodaroli\footnote{\texttt{elisabetta.chiodaroli@unipi.it}} 
\and Eduard Feireisl\footnote{\texttt{feireisl@math.cas.cz}}  \and Ond\v rej Kreml\footnote{\texttt{kreml@math.cas.cz}} \and Simon Markfelder\footnote{\texttt{simon.markfelder@uni-konstanz.de}}}

\date{}

\maketitle

\medskip

\centerline{$^{\ast}$ Dipartimento di Matematica, Universit\` a di Pisa,} 
\centerline{ Via F. Buonarroti 1/c, 56127, Pisa, Italy}

\medskip

\centerline{$^{\dagger}$ $^{\ddagger}$ Institute of Mathematics of the Academy of Sciences of the Czech Republic,}
\centerline{\v Zitn\' a 25, CZ-115 67 Praha 1, Czech Republic}

\medskip

\centerline{$^{\S}$ University of Konstanz, Department of Mathematics and Statistics,} 
\centerline{Post office box: 199, 78457 Konstanz, Germany}

\date{}

\maketitle

\begin{abstract}
	
Convex integration has revealed that the Euler system of gas dynamics is ill-posed in the class of weak solutions even if the entropy inequality is imposed as an additional constraint. A natural question arises, namely, if a physically relevant solution can be selected by maximizing the entropy production rate. Firstly, we present an example of Riemann initial data in 2--D, for which the standard self-similar solution fails to satisfy the maximal entropy production principle. Hence, maximizing the entropy production rate rules out the 1--D self-similar solution which intuitively seems to be the physically relevant solution in this context. Secondly, we show for a large class of initial data that there exist entropy admissible weak solutions with an arbitrary (non--decreasing) total entropy profile.

\end{abstract}

\bigskip


\bigskip

\section{Introduction}

The Euler system of gas dynamics is a mathematical model describing the time evolution of a compressible inviscid fluid taking into account changes in the temperature. From the mathematical viewpoint, it represents an iconic example of a system of non--linear conservation laws, see e.g.~Part IV of the monograph by Benzoni-Gavage and Serre \cite{BenSer} or Dafermos \cite{D4a}.

\subsection{Euler system}

The Euler system uses the language of continuum mechanics.  
The time evolution of the phase variables, the mass density $\vr= \vr(t,x)$, the (absolute) temperature $\vt = \vt(t,x)$, and the velocity 
$\vv = \vv(t,x)$, is described by means of a system of field equations:
\begin{align}
	\partial_t \vr + \Div (\vr \vv) &= 0, \label{E1}\\ 
	\partial_t (\vr \vv) + \Div (\vr \vv \otimes \vv) + \Grad p(\vr, \vt) &= 0, \label{E2}\\
	\partial_t \left( \frac{1}{2} \vr |\vv|^2 + \vr e(\vr, \vt) \right) + 
	\Div \left[ \left( \frac{1}{2} \vr |\vv|^2 + \vr e(\vr, \vt) + p(\vr, \vt) \right) \vv \right] &=0, \label{E3}
\end{align}
expressing the physical principles of conservation of mass, linear momentum, and energy, respectively.
For definiteness, we suppose that the pressure $p=p(\vr, \vt)$ and the internal energy $e=e(\vr, \vt)$ obey the standard Boyle--Mariotte law,
\begin{equation} \label{E4} 
	p(\vr, \vt) = \vr \vt,\ e(\vr, \vt) = c_v \vt,\ c_v > 0,
	\end{equation}
where the constant $c_v$ represents the specific heat at constant volume.
The problem is completed by prescribing the initial conditions
\begin{equation} \label{E5}
\vr(0, \cdot) = \vr_0 ,\ \vt(0, \cdot) = \vt_0,\ \vv(0, \cdot) = \vv_0,
\end{equation}
as well as some relevant boundary conditions depending on the choice of the 
physical domain occupied by the fluid. 

On the one hand, the Euler system is well--posed, locally in time, in the class of sufficiently regular initial data, see e.g.~Benzoni-Gavage and Serre \cite[Part IV] {BenSer} or Schochet \cite[Theorem 1]{SCHO1}. On the other hand, strong solutions may develop singularities in the form of shock waves appearing in a finite time, see e.g.~the monograph of Smoller 
\cite{SMO}. More recently, Buckmaster et al.~\cite{BUCLGS} showed the 
existence of self-imploding singularities for the simplified isentropic version of the Euler system. 

In addition to \eqref{E1}--\eqref{E3}, \emph{entropy admissible} weak (distributional) solutions  
are required to satisfy the entropy inequality
\begin{equation} \label{entro}
	\partial_t (\vr s(\vr, \vt)) + \Div( \vr s(\vr, \vt) \vv) \geq 0,\ \mbox{with entropy}\ 
	s(\vr, \vt) = c_v \log \vt - \log \vr, 
\end{equation}
in the sense of distributions.

\subsection{Ill--posedness in the class of entropy admissible weak solutions}

The recent developments of the theory of convex integration in the framework 
of fluid dynamics, starting with the seminal works of De Lellis and Sz{\'e}kelyhidi \cite{DelSze09} and \cite{DelSze3},
revealed several rather surprising facts concerning well--posedness of the Euler system even in the class of \emph{entropy admissible} 
weak solutions. The ill--posedness of the Riemann problem in the 
physically relevant 2/3--D geometry was shown by Klingenberg et al.~\cite[Theorem 1.1]{KlKrMaMa}, see also Al Baba et al.~\cite{ABKlKrMaMa}. 
In \cite{FeKlKrMa}, a large class of so-called ``wild'' initial data that give rise to infinitely many entropy admissible weak solutions was identified. Finally, it was shown in \cite{ChFe2023} that wild data are dense in the $L^p-$topology. 

\subsection{Entropy production rate as an admissibility criterion} 
 
In view of Schwartz representation theorem, the entropy inequality 
\eqref{entro} can be interpreted as 
\begin{equation*} 
	\partial_t (\vr s(\vr, \vt)) + \Div( \vr s(\vr, \vt) \vv) = \sigma, 
\end{equation*}
where $\sigma \in L^\infty(0,T; \mathcal{M}^+(\Ov{\Omega}))$ is the 
entropy production rate. The symbol $\mathcal{M}^+$ denotes the set of 
all non--negative Borel measures defined on the closure of the fluid domain 
$\Omega$.

DiPerna \cite{Dip2} proposed a partial ordering on the class of entropy admissible solutions:
\begin{equation} \label{DiP1}
(\vr^1, \vt^1, \vv^1) \ \prec_{\rm DiP} \ (\vr^2, \vt^2, \vv^2)
\ \Leftrightarrow\ \intO{ \vr^1 s(\vr^1, \vt^1) (t, \cdot) } 
\leq  \intO{ \vr^2 s(\vr^2, \vt^2) (t, \cdot) } \ \mbox{for a.a.}\ 
t \in (0,T). 
\end{equation}
We define as \emph{DiPerna admissible} weak solutions those weak solutions that are \emph{maximal} with respect to $\prec_{\rm DiP}$. 

Dafermos in \cite{Dafer} advocates a refined version of $\prec_{\rm DiP}$, specifically  
\begin{align} 
	(\vr^1, \vt^1, \vv^1) \ &\prec_{\rm Daf} \ (\vr^2, \vt^2, \vv^2)
	\br &\Leftrightarrow \br \mbox{there exists}\ \tau\geq 0 
	\ \mbox{such that}\ 
(\vr^1, \vt^1, \vv^1)(t, \cdot) &= (\vr^2, \vt^2, \vv^2)(t, \cdot) \ \mbox 
{for}\ 0 \leq t < \tau,	\br
 \frac{{\rm d}^+}{\dt}		\intO{ \vr^1 s(\vr^1, \vt^1) (\tau, \cdot) } 
	&<  \frac{{\rm d}^+}{\dt}\intO{ \vr^2 s(\vr^2, \vt^2) (\tau, \cdot) }.  \label{Dafer1}
\end{align}
Similarly, a weak solution is called \emph{Dafermos admissible} or \emph{entropy rate admissible} if it 
is maximal with respect to the partial ordering $\prec_{\rm Daf}$. 
It is worth noting that both \eqref{DiP1} and \eqref{Dafer1} were proposed 
for closed systems with energy conservative boundary conditions. 

Our aim is to shed some light on the relevance of this type of admissibility 
criteria for a correct choice of solution to the Euler system. 
We discuss two different problems.

\subsubsection{Admissibility of self--similar solutions for the Riemann problem}

We consider the 1--D (piecewise constant) Riemann initial data in the 2--D domain $\R^2$, extended to be constant in the $x_1$ variable: 
\begin{equation}\label{eq:R_data}
	(\vr_0,\vt_0,\vv_0) = \left\{
	\begin{array}{l}
		(\vr_-,\vt_-, \vv_-)\ \mbox{ if } x_1 \in \R, \ x_2<0 \\
		(\vr_+, \vt_+, \vv_+)\ \mbox{ if } x_1 \in \R,\ x_2>0.
	\end{array}
	\right. 
\end{equation}

Since these initial data are one-dimensional, the classical theory of hyperbolic conservation laws yields the existence of a unique 1--D solution 
(independent of $x_1$) to problem \eqref{E1}--\eqref{E3} with initial data \eqref{eq:R_data} satisfying the entropy inequality \eqref{entro}. This solution consists of constant states connected by either shocks, rarefaction waves or contact discontinuities. We will call these solutions \emph{1--D self-similar solutions} and we refer to \cite{ABKlKrMaMa}, \cite{KlKrMaMa}, and \cite{SMO}, for the conditions on the initial data giving rise to different types of 1--D self-similar solutions. See also Dafermos \cite{D4a} for more details about the 1--D Riemann problem for systems of conservation laws.

On the unbounded domain $\R^2$, we modify the definition of Dafermos' ordering \eqref{Dafer1} by introducing the following order relation among weak solutions:
\begin{align} 
	(\vr^1, \vt^1, \vv^1) \ &\prec \ (\vr^2, \vt^2, \vv^2)
	\br &\Leftrightarrow \br \mbox{there exists}\ \tau\geq 0 
	\ \mbox{such that}\ 
	(\vr^1, \vt^1, \vv^1)(t, \cdot) &= (\vr^2, \vt^2, \vv^2)(t, \cdot) \ \mbox 
	{for}\ 0 \leq t < \tau, \br \mbox{there exists}\ \underline{L} 
	\ \mbox{such that}\ 
	\frac{{\rm d}^+}{\dt} \int_{-L}^L \int_{-L}^L \vr^1 s(\vr^1, \vt^1) (\tau, \cdot) \ \D x_1 \D x_2 
	&<  \frac{{\rm d}^+}{\dt}  \int_{-L}^L \int_{-L}^L \vr^2 s(\vr^2, \vt^2) (\tau, \cdot) \ \D x_1 \D x_2\ \mbox{for all}\ L \geq \underline{L}.  \label{Dafer2}
\end{align}

\begin{Definition}[\bf Entropy rate admissibility for the unbounded domain $\R^2$] \label{Di1}
In the context of the initial data \eqref{eq:R_data}, we say a weak solution 
$(\vr, \vt, \vv)$ is \emph{Dafermos admissible} or \emph{entropy rate admissible} if it is 
maximal with respect to the partial order relation $\prec$. 
\end{Definition}

\bRemark{Rem_periodic}
The restriction to a bounded square $[-L,L]^2\subset \R^2$ as done in the definition of the order $\prec$, see \eqref{Dafer2}, is one possibility to overcome the issue of infinite total entropy on the unbounded domain $\R^2$. Another option is to periodize as explained in \cite[Section 6]{ChiKre} for the isentropic case. The same strategy can also be carried out in the case of the full Euler system \eqref{E1}--\eqref{E3} studied in this paper. 
\eR

Our first result, formulated and proved in Section \ref{R} (see Theorem \ref{Tmain_1}) shows the existence of some 1--D Riemann initial data in $\R^2$ such that the associated 1--D self-similar solution \emph{is not} entropy rate admissible. As a consequence of the construction, in this case the 1--D self-similar solution also fails to be DiPerna admissible.

\subsubsection{Admissible forms of the entropy production rate in the class of weak solutions}

For the second problem discussed in this paper we consider the Euler system \eqref{E1}--\eqref{E3} in a bounded (regular) spatial domain $\Omega \subset \R^3$, supplemented with the impermeable wall boundary conditions 
\begin{equation} \label{imw}
	\vv \cdot \vc{n}|_{\partial \Omega} = 0.
\end{equation}

As shown in \cite{FeiLukYu}, the problem admits a generalized dissipative 
measure--valued (DMV) solution for \emph{any} finite energy initial data. 
Moreover, rather surprisingly, 
the entropy rate admissible DMV solutions, specifically maximal with respect to $\prec_{\rm Daf}$, must be weak solutions, see \cite[Theorem 3.2]{FeiLukYu}. The proof of this result is based on a special construction of the entropy production rate $\sigma$, with jumps - concentrations 
in the form of a Dirac mass sitting on a specific set of times. Some natural 
questions arise: to which extent is this construction pertinent to DMV solutions and do there exist or not weak solutions with arbitrary entropy profile? 

Given a non--decreasing function $\widetilde{S}: [0,T]$, $S(0+) = 0$, we show that there always exist initial data $(\vr_0, \vt_0, \vv_0)$ and a global in time entropy admissible weak solution $(\vr, \vt, \vv)$ to the Euler system satisfying   
\begin{equation*} 
\intO{ \vr s (\vr, \vt) (t, \cdot) } = S_0 + c_v \widetilde{S}(t) \intO{ \vr_0 (\cdot)}
\quad \ \mbox{for a.a.}\ t \in (0,T),
\end{equation*}
see Theorem \ref{TEP1}. The proof is given in Section \ref{EP} and is based on a
``discontinuous'' convex integration technique developed in \cite{AbbFei}.

\bigskip 

\hrule 

\bigskip

The paper is organized as follows. Section \ref{R} is devoted to Riemann problems and incompatibility of 1--D self--similar solutions with the entropy rate admissibility. In Section \ref{EP}, we identify the initial data that give rise to weak solutions with a given entropy production profile. 
The paper is concluded with a short discussion in Section \ref{C}.

\section{Riemann problem}
\label{R}

We consider the Riemann problem for the Euler system \eqref{E1}--\eqref{E3}
in the space time--cylinder $(0,T) \times \R^2$, with the 1--D initial data 
\eqref{eq:R_data}. As mentioned above, 
the problem admits infinitely many entropy admissible weak solutions 
as long as the initial data contain a shock, see \cite{ABKlKrMaMa}, \cite
{KlKrMaMa}. On the contrary, if the 1--D self-similar solution consists only of rarefaction waves and does not contain neither shocks nor a contact discontinuity, this solution is in fact unique in the class of entropy admissible weak solutions to the multidimensional problem, see Chen and Chen \cite{CheChe}, and \cite{FeKrVa}.
Interestingly, the question whether the 1--D self-similar solution to \eqref{E1}--\eqref{E3} with initial data \eqref{eq:R_data} containing a contact discontinuity but not containing a shock is unique among entropy admissible solutions to the multidimensional problem is still open.

In the context of the isentropic Euler equations, Krupa and Sz\'ekelyhidi in \cite{KRSZ} addressed the analogous problem in the regime when the 1--D self-similar solution consists of a single contact discontinuity, and they showed non-uniqueness in the class of bounded, admissible weak solutions for contact discontinuity initial data in more than one space dimension. Their proof relies on a computer-assisted framework that  exploits the pressure law as an additional degree of freedom. Very recently, Horimoto \cite{HOR} established the existence of infinitely many admissible weak solutions for a subclass of contact discontinuity Riemann data with a general strictly increasing pressure law. 

The non-uniqueness results quoted above highlight the problem of finding a suitable selection principle which is supposed to pick the physically relevant solution among all entropy admissible weak solutions. 
Apparently, the entropy inequality alone does not provide the answer, so other selection criteria have to be considered.

In \cite{ChiKre}, the concept of maximal dissipation of total energy for the isentropic Euler system was studied. As mentioned above, the idea goes back to the work of Dafermos \cite{Dafer}, DiPerna \cite{Dip2}, and also Hsiao \cite{Hsiao}. 
It was shown in \cite{ChiKre} that there exist Riemann initial data to the isentropic Euler system such that the associated 1--D self-similar solution is not energy dissipation rate admissible, meaning there exists another solution which dissipates the total energy at a higher rate. This result was further strengthened by the fourth author \cite{Mark}, who showed that even a local version of the energy dissipation rate criterion does not select the 1--D self-similar solution as the best. 

In the context of the isentropic Euler equations, another criterion regarding the action has been proposed by Gimperlein et al.~\cite{Slemrod}, the so-called \emph{least action criterion}. Similarly to the aforementioned results on the energy dissipation rate, it was shown in \cite{MarPel} that the least action criterion does not single out the 1--D self-similar solution. 

\subsection{Entropy rate admissibility for the Euler system of gas dynamics}

Although the total energy of the isentropic Euler system 
is somehow interpreted, mostly by mathematicians, as entropy, the 
weak solutions of the isentropic Euler system do not solve 
the (full) Euler system \eqref{E1}--\eqref{E3} unless their total energy as well as entropy are constant. Seen from this viewpoint, the Euler system 
\eqref{E1}--\eqref{E3} is much closer to the real world problems, typically 
for gases. Still the entropy rate admissibility criterion introduced 
in Definition \ref{Di1} fails to identify the 1--D self-similar solution as 
admissible in the 2--D setting. 

Our first main result reads as follows.

\begin{Theorem}[\bf The 1--D self-similar solution is not entropy rate admissible] \label{Tmain_1}
Let $c_v = \frac 32$. There exist Riemann initial data \eqref{eq:R_data} such that the 1--D self-similar solution to the Euler system \eqref{E1}--\eqref{E3}, with \eqref{entro} in $(0,T) \times \R^2$ is not entropy rate admissible in the sense specified in Definition \ref{Di1}.
\end{Theorem}

Theorem \ref{Tmain_1} can be seen as a non-isentropic version of the aforementioned analogous result obtained in \cite{ChiKre}.

\bRemark{Rem_Di Perna}
Theorem \ref{Tmain_1} relies on the construction via convex integration of a solution (in fact infinitely many) which dissipates
the total entropy at a higher rate than the self--similar solution. 
As an outcome of the structure of the competing solution(s) constructed, it is easy to check that, for the same choice of Riemann data, the 1--D self-similar solution fails to be DiPerna admissible, too. 
\eR

\begin{proof}[Proof of Theorem \ref{Tmain_1}]
We provide a specific example. We take
\begin{equation*}
\begin{split}
    (\vr_-,\vv_-,\vt_-) &= (1,(0,0),2), \\
    (\vr_+,\vv_+,\vt_+) &= \left(10,(0,-100),\frac{1}{10}\right).
\end{split}
\end{equation*}
For our purposes it is more suitable to work with pressure instead of temperature, so using \eqref{E4} we reformulate these data as
\begin{equation}\label{eq:Bad_Rdata}
\begin{split}
    (\vr_-,\vv_-,p_-) &= (1,(0,0),2), \\
    (\vr_+,\vv_+,p_+) &= (10,(0,-100),1).
\end{split}
\end{equation}
It is not difficult to check that these data satisfy the assumptions of \cite[Proposition 2.1]{ABKlKrMaMa} and therefore the 1--D self-similar solution $(\vr_s,\vv_s,p_s)$ which emanates from these initial data consists of two shocks and a contact discontinuity. According to \cite[Proposition 2.1]{ABKlKrMaMa}, the intermediate states of the 1--D self-similar solution are given by $p_M$ which is the unique solution to the equation 
\begin{equation*}
			-\sqrt{2\,c_v}\left(\frac{p_M - p_-}{\sqrt{\vr_- (p_- + (2c_v + 1) p_M)}} + \frac{p_M - p_+}{\sqrt{\vr_+ (p_+ + (2c_v + 1) p_M)}}\right) = v_{+,2} - v_{-,2},
\end{equation*}
and its approximate value for \eqref{eq:Bad_Rdata} is 
\begin{equation}\label{eq:pm}
p_M \sim 7700.164.    
\end{equation}
The second component of the velocity of the intermediate states is given by
\begin{equation*}
    v_{M,2} = v_{-,2} - \sqrt{2\,c_v}\frac{p_M - p_-}{\sqrt{\vr_- (p_- + (2c_v + 1) p_M)}},
\end{equation*}
which yields
\begin{equation}\label{eq:vm}
\vv_M \sim (0,-75.972).
\end{equation}
The densities of the intermediate states are given by
\begin{equation*}
    \vr_{M\pm} = \vr_\pm \frac{p_\pm + (2c_v + 1) p_M}{p_M + (2c_v + 1) p_\pm},
\end{equation*}
giving rise to approximate values 
\begin{equation}\label{eq:rm}
\vr_{M-} \sim 3.996, \qquad \vr_{M+} \sim 39.981.
\end{equation}
The shock speeds are given by 
\begin{equation*}
    \sigma_\pm(\vr_\pm - \vr_{M\pm}) = \vr_\pm v_{\pm,2} - \vr_{M\pm}v_{M,2},
\end{equation*}
and their approximate values are
\begin{equation}\label{eq:shocks}
    \sigma_- \sim -101.329, \qquad \sigma_+ \sim -67.957,
\end{equation}
while the speed of the contact discontinuity $\sigma_M$ is equal to the second component of the intermediate velocity $\vv_M$, see \eqref{eq:vm}.

Since the 1--D self-similar solution is piecewise constant, it is not difficult to express the entropy production rate. It holds
\begin{equation*}
\begin{split}
    D_L[\vr_s,p_s] &:= \frac{{\rm d}^+}{\dt} \int_{-L}^L \int_{-L}^L \vr_s s\left(\vr_s, \frac{p_s}{\vr_s}\right) (t, \cdot) \ \D x_1 \D x_2 \\ 
    &= 2L\Big(\sigma_-(\vr_-s_- - \vr_{M-}s_{M-}) + \sigma_M(\vr_{M-}s_{M-} - \vr_{M+}s_{M+}) + \sigma_+(\vr_{M+}s_{M+} - \vr_+s_+)\Big),
\end{split}
\end{equation*}
where we denoted $s_{\pm} = s\left(\vr_\pm,\frac{p_\pm}{\vr_\pm}\right) = \log \frac{p_\pm^{c_v}}{\vr_\pm^{c_v+1}}$ and $s_{M\pm} = s\left(\vr_{M \pm},\frac{p_M}{\vr_{M\pm}}\right) = \log \frac{p_M^{c_v}}{\vr_{M\pm}^{c_v+1}}$. Note that $D_L[\vr_s,p_s]$ is in fact independent of time.

This enables us to calculate directly the value of the entropy production rate of the 1--D self-similar solution to be 
\begin{equation*}
    D_L[\vr_s,p_s] \sim 2L\cdot(-1661.456).
\end{equation*}
We emphasize that all quantities related to the studied 1--D self-similar solution are solutions to algebraic equations, in particular every value depends continuously on $p_M$, so it surely holds
\begin{equation} \label{eq:rate-1D}
    -1662 < \frac{D_L[\vr_s,p_s]}{2L} < -1661.
\end{equation}

Next, we will construct another solution $(\vr_c,\vv_c,p_c)$ to \eqref{E1}--\eqref{E3} with the same initial data \eqref{eq:Bad_Rdata} satisfying \eqref{entro}. This solution - in fact there are infinitely many of them - will be a so-called $1$-fan solution, which stems from a single piecewise constant $1$-fan subsolution, as introduced in \cite[Section 4]{ABKlKrMaMa}.

We take the intermediate density of the $1$-fan subsolution to be
\begin{equation}\label{eq:rho1}
    \rho_1 = 14.
\end{equation}
Solving the Rankine-Hugoniot conditions for the subsolution (see \cite[Proposition 4.4]{ABKlKrMaMa}) gives rise to the following values:
\begin{equation}\label{eq:othervalues}
\begin{array}{lll}
    \mu_- \sim -91.620, & \mu_+ \sim -47.765, & \beta \sim -85.076, \\
    p_1 \sim 4578.655, & C_1 \sim 7528.076, & \gamma \sim -3703.705.
\end{array}
\end{equation}
It is not difficult to check, that all the inequalities from \cite[Proposition 4.4]{ABKlKrMaMa} are satisfied as well. Since all solutions which are produced using this subsolution via the convex integration procedure share the values of density and pressure with the subsolution, we can calculate the entropy production rate of every such solution using only the subsolution. Similarly to the case of the 1--D self-similar solution, the entropy production rate is a constant given by
\begin{equation*}
\begin{split}
    D_L[\vr_c,p_c] &:= \frac{{\rm d}^+}{\dt} \int_{-L}^L \int_{-L}^L \vr_c s\left(\vr_c, \frac{p_c}{\vr_c}\right) (t, \cdot) \ \D x_1 \D x_2 \\ 
    &=  2L\Big(\mu_-(\vr_-s_- - \vr_1s_1) +\mu_+(\vr_1s_1 - \vr_+s_+)\Big),
\end{split}
\end{equation*}
where $s_1 = s\left(\vr_1,\frac{p_1}{\vr_1}\right) = \log \frac{p_1^{c_v}}{\vr_{1}^{c_v+1}}$. Plugging in the values \eqref{eq:rho1}, \eqref{eq:othervalues}, we obtain
\begin{equation*}
    D_L[\vr_c,p_c] \sim 2L\cdot(867.268).
\end{equation*}
Since all values in \eqref{eq:othervalues} are given as solutions to algebraic equations, they depend continuously on given data, therefore it surely holds 
\begin{equation} \label{eq:rate-wild}
    867 < \frac{D_L[\vr_c,p_c]}{2L} < 868.
\end{equation}
Comparing \eqref{eq:rate-1D} with \eqref{eq:rate-wild}, we find
\begin{equation*}
    D_L[\vr_c,p_c] > D_L[\vr_s,p_s],
\end{equation*}
which proves Theorem \ref{Tmain_1}.
\end{proof}

\bRemark{Rem_cv}

In fact, the same initial data and the same choice of $\vr_1$ in the
proof of \ref{Tmain_1} produce a similar counterexample with $c_v = 1$ and we observed
the same conclusion to hold with any choice of $c_v \in [1, \frac32]$ we tried, therefore
we conjecture that Theorem \ref{Tmain_1} holds for any $c_v \in [1,
\frac{3}{2}]$.
\eR

\bRemark{Rem_dimension}
The statement of Theorem \ref{Tmain_1} holds in any space dimension larger than $1$ with obvious modifications to the definition of entropy rate admissibility \eqref{Dafer2} to accommodate higher dimensions. It is enough to extend all solutions in the proof of Theorem \ref{Tmain_1} by constants to higher dimensions.
\eR

\section{Weak solutions with arbitrary entropy profile}
\label{EP}

In contrast to Section \ref{R}, we now consider the Euler system \eqref{E1}--\eqref{E3} in a bounded 
domain $\Omega \subset \R^3$ supplemented with the impermeable wall boundary 
conditions \eqref{imw}. Given an entropy admissible weak solution $(\vr, \vt, \vv)$, we set 
\begin{equation} \label{EP1}
S(t) = \intO{ \vr(t, \cdot) s(\vr(t, \cdot), \vt(t, \cdot)) } 
\ \mbox{for a.a.}\ t \in (0,T).
\end{equation}	
In accordance with the entropy inequality \eqref{entro}, the function 
can be identified for a.a.~$t \in (0,T)$ with a non--decreasing function of 
$t$.

Next, we introduce a family of entropy profiles 
\begin{equation} \label{p15}
	\mathfrak{S}[\delta] = \left\{ \widetilde{S} \ \mbox{non--decreasing in}\ [0,T] \ \Big| \ \widetilde{S}(t) = 0 \ \mbox{for}\ 0 \leq t < \delta ,\ 
\widetilde{S}(T-) < \infty \right\}. 	
\end{equation}	
Equivalently, we may say that the profiles belonging to 
$\mathfrak{S}[\delta]$ are bounded, vanishing in the interval $[0, \delta)$, 
with distributional derivative 
$\widetilde{S}' \in \mathcal{M}^+ [0,T]$ -- a non-negative Borel measure. 

Our second main result reads:

\begin{Theorem}[\bf Solutions with given entropy profile] \label{TEP1}
Let $\Omega$ be a bounded domain in $\R^3$ that admits a decomposition
\begin{equation*} 
	\Ov{\Omega} = \cup_{i=1}^N \Ov{Q}_i,\ Q_i \cap Q_j = \emptyset \ \mbox{for}\ i \ne j, \ Q_i \ \mbox{bounded domains.}	
\end{equation*}
Suppose the initial density $\vr_0$, and the initial temperature $\vt_0$ in 
\eqref{E5} are piecewise constant, specifically 
\begin{equation*} 
	\vr_0 |_{Q_i} = \vr_0^i - \mbox{a positive constant},\ \vt_0|_{Q_i} = \vt_0^i - \mbox{a positive constant},\ i = 1, \dots, N.
\end{equation*}
Let $c_v = \frac{3}{2}$, the total entropy profile $\widetilde{S} \in 
\mathfrak{S}[\delta]$, and $0 < \ep < \delta$ be given.

Then there exists an initial velocity $\vv_0 \in L^\infty(\Omega; \R^3)$ such that the Euler system \eqref{E1}--\eqref{E3}, \eqref{E5}, \eqref{imw} admits infinitely many entropy admissible weak solutions satisfying
\begin{equation} \label{EP2}
\frac{\D }{\dt} \intO{ \vr s} = c_v\widetilde{S}'(t + \ep) M_0 
\ \mbox{in}\ \mathcal{D}'(0, T- \ep),\ \mbox{where}\ \ M_0 = \intO{ \vr_0 }.
\end{equation} 

\end{Theorem}	

The remaining part of this section is devoted to the proof of Theorem \ref{TEP1}. We use the ansatz from \cite{FeKlKrMa},  combined with the ``discontinuous'' variant of convex integration from \cite{AbbFei}.

\subsection{Convex integration ansatz}

Given the initial data $\vr_0$, we fix the density to be constant in each 
$Q_i$, specifically, 
\begin{equation} \label{EP3}
\vr(t,x) = \vr_0^i \ \mbox{whenever}\ t \in [0,T], \ x \in Q_i,\ i=1, \dots, N.
\end{equation}	

The velocity is replaced by the momentum $\vm = \vr \vv$ satisfying 
\begin{equation} \label{p1}
	\vm \in L^1((0,T) \times \Omega; \R^3),\ \Div \vm = 0,\ \vm \cdot \vc{n}|_{\partial \Omega} = 0, 	
\end{equation}
specifically
\begin{equation} \label{p2}	
	\int_0^T \intO{ \vm \cdot \Grad \varphi } \dt = 0 \ \mbox{for any}\ \varphi \in C^1([0,T] \times \Ov{\Omega}).
\end{equation}

\subsubsection{Equation of continuity}

Obviously, the equation of continuity \eqref{E1} is satisfied being reduced to the identity
\begin{equation} \label{p3}
	\int_{\tau_1}^{\tau_2} \intO{\vr_0 \partial_t \varphi} = \left[ \intO{ \vr_0 \varphi } \right]_{t = \tau_1}^{t = \tau_2},\ 
	0 \leq \tau_1 \leq \tau_2 \leq T 
\end{equation}
for any $\varphi \in C^1([0,T] \times \Ov{\Omega})$. In particular, the total mass 
\[
M_0 = \intO{ \vr_0 }
\]
is a constant of motion.

\subsubsection{Entropy balance}

The temperature $\vt_0$ is a piecewise constant function, 
\begin{equation*} 
0 < \underline{\vt} \leq \vt^i_0 \leq \Ov{\vt},\ i = 1,\dots, N  	
\end{equation*}
Given $\widetilde{S} \in \mathfrak{S}[\delta]$, we set  
\begin{equation} \label{p5}
\vt(t,x) = \exp \Big( \widetilde{S}(t) \Big) \vt^i_0 \ \mbox{for}\ 
t \in [0,T],\ x \in Q_i, \ i = 1,\dots, N. 
\end{equation}
The associated internal energy reads 
\begin{equation*} 
(\vr e)(t,x) = c_v \vr^i_0 \vt^i_0 \exp \Big( \widetilde{S}(t) \Big) 
\ \mbox{for}\ t \in [0,T],\ x \in Q_i,\ i =1, \dots, N,
\end{equation*} 
with the entropy 
\begin{equation} \label{p7}
s(t,x) = c_v \log (\vt) - \log(\vr) = c_v \widetilde{S}(t) + c_v 
\log(\vt^i_0) - \log (\vr^i_0) \ \mbox{for}\ t \in [0,T],\ x \in Q_i,\ 
i = 1,\dots, N.
\end{equation}

Moreover, we compute the time derivative,

\begin{align*} 
\int_{\tau_1}^{\tau_2} \intO{ (\vr s) \partial_t \varphi} \dt &= \left[ \intO{ \Big( c_v \vr_0 \log(\vt_0)	- \vr_0 \log(\vr_0) \Big) \varphi } \right]_{t = \tau_1+}^{t = \tau_2 -} \br 
&+ \int_{\tau_1}^{\tau_2} \intO{ c_v \vr_0 \widetilde{S}(t) \partial_t \varphi } \dt
\end{align*}	
for any $\varphi \in C^1([0,T] \times \Ov{\Omega})$. 

It follows from \eqref{p7} that the total entropy is independent of $x$ on any $Q^i$. Thus we get the desired entropy balance 
\begin{equation} \label{EP15}
\partial_t (\vr s) + \Div (s \vm) = c_v \vr_0 \widetilde{S}'(t) 	
\end{equation}
as long as the momentum $\vm$ belongs to the class \eqref{p1}, with \eqref{p2} refined to 
\begin{equation} \label{p10}
\int_0^T \int_{Q_i} \vm \cdot \Grad \varphi \dx \dt = 0 \ \mbox{for any}\ \varphi \in C^1(\Ov{Q}_i ),\ i = 1, \dots, N.	
\end{equation}	
Condition \eqref{p10} means $\Div \vm = 0$ in $\Omega$, and 
$\vm \cdot \vn|_{Q_i} = 0$ for any $i = 1,\dots, N$.

\subsubsection{Momentum and total energy balance}

Similarly to \cite{FeKlKrMa}, the
momentum balance \eqref{E2} will be solved on any $Q = Q_i$, $i = 1,\dots, N$ by convex integration:
\begin{equation*} 
\Div \vm = 0,\ 	
\partial_t \vm + \Div \left( \frac{ \vm \otimes \vm }{\vr^i_0} - \frac{1}{3} \frac{|\vm|^2}{\vr^i_0} \mathbb{Id} \right) = 0 \ \mbox{in}\ (0,T) \times Q_i,\ i = 1,\dots, N,
\end{equation*}	
with the ``do--nothing'' boundary conditions, meaning without any restriction on the boundary behavior of test functions. More specifically, besides 
\eqref{p10}, we require 
\begin{equation*} 
\int_0^T \int_{Q_i} \left[ \vm \cdot \partial_t \bfphi + 
\left( \frac{\vm \otimes \vm}{\vr^i_0} - \frac{1}{3} \frac{|\vm|^2}{\vr^i_0} 
\mathbb{I} \right) : \Grad \bfphi \right] \dx \dt = - 
\int_{Q_i} \vm_0 \cdot \bfphi(0, \cdot) \dx ,\ i = 1, \dots, N,
\end{equation*}	
for any $\bfphi \in C^1_c([0,T) \times \Ov{Q}_i)$. Moreover, we prescribe 
the kinetic energy 
\begin{equation*} 
\frac{1}{2} \frac{|\vm|^2}{\vr^i_0} = \Lambda - \frac{3}{2} \exp\Big( \widetilde{S}(t) \Big) \vr^i_0 \vt^i_0 \ \mbox{in}\ (0,T) \times Q_i, 
\end{equation*}
where $\Lambda > 0$ is a suitable constant independent of $i$ to be determined below.

Formally, the total energy  
\begin{align*} 
\frac{1}{2} \frac{|\vm|^2}{\vr^i_0} + \frac{3}{2} \exp \Big( \widetilde{S}(t) \Big) \vr^i_0 \vt^i_0 = \Lambda \ \mbox{in}\ Q_i
\end{align*}
is constant in $\Omega$. In view of 
\eqref{p10}, the total energy balance \eqref{E3} is automatically satisfied. 

\subsection{Convex integration}

In the preceding part, we have reduced the proof of 
Theorem \ref{TEP1} to the solvability of the problem
\begin{align} 
\int_0^T \int_{Q_i} \vm \cdot \Grad \varphi \dx \dt &= 0 \ \mbox{for any}\ \varphi \in C^1(\Ov{Q}_i ) ,\br
\int_0^T \int_{Q_i} \left[ \vm \cdot \partial_t \bfphi + 
\left( \frac{\vm \otimes \vm}{\vr^i_0} - \frac{1}{3} \frac{|\vm|^2}{\vr^i_0} 
\mathbb{I} \right) : \Grad \bfphi \right] \dx \dt &= - 
\int_{Q^i} \vm_0 \cdot \bfphi(0, \cdot) \dx \  
\mbox{for any}\ \bfphi \in C^1_c([0,T) \times \Ov{Q}_i) \br
\frac{1}{2} \frac{|\vm|^2}{\vr^i_0} &= \Lambda - \frac{3}{2} \exp\Big( \widetilde{S}(t) \Big) \vr^i_0 \vt^i_0 \ \mbox{a.e. in}\ (0,T) \times Q_i 
	\label{problem}	
\end{align}	 
for any $i = 1,\dots, N$. 

Problem \eqref{problem} can be solved by the method of convex integration.
In contrast with the ``standard'' ansatz, the desired kinetic energy profile 
\[
\Lambda - \frac{3}{2} \exp\Big( \widetilde{S}(t) \Big) \vr^i_0 \vt^i_0 
\]
is not continuous but definitely falls into the category of Riemann integrable 
functions considered in \cite{AbbFei}. Consequently, the conclusion 
of Theorem \ref{TEP1} follows from \cite[Theorems 2.1, 2.4]{AbbFei}. As a matter of fact, these results guarantee the existence of only one weak solution. However, using the procedure detailed in \cite[Section 13.6]{Fei2016}, this solution can be used as a subsolution in the convex integration process to obtain infinitely many solutions as claimed.

\section{Conclusion}
\label{C}

The previous discussion reveals a fundamental question: to which extent 
the selection criteria based on the second law of thermodynamics, specifically maximization of the entropy production rate, are relevant for the 
Euler system of gas dynamics? Summarizing we know that: 
\begin{itemize}
	\item the dissipative measure valued (DMV) solutions that are entropy rate admissible are necessarily weak solutions, see \cite{FeiLukYu}; 
	\item there exist 1--D Riemann data for which the 1--D self--similar solution is not entropy rate admissible in the 2--D setting, see Theorem~\ref{Tmain_1}; 
	\item the weak solutions emanating from wild initial data obtained by the available methods based on convex integration are not entropy rate admissible, which is a consequence of Theorem~\ref{TEP1}. 
\end{itemize}	 

As a natural closure of the set of entropy admissible weak solutions 
of the Euler system is represented by the class of DMV solutions, we conjecture that the entropy rate admissible weak solutions possibly \emph{do not exist} for a generic set of initial data.

\section{Acknowledgments}
E.C.~acknowledges financial support by the project PRIN 2022T9K54B. The work of E.F.~was partially supported by the Czech Sciences Foundation (GA\v CR), Grant Agreement 24--11034S. The work of O.K.~was supported by the Praemium Academiae of \v S. Ne\v casov\' a. The Institute of Mathematics of the Czech Academy of Sciences is supported by RVO:67985840. E.F.~and O.K.~are members of the Ne\v cas Center for Mathematical Modelling. S.M.~acknowledges financial support from the Deutsche Forschungsgemeinschaft (DFG, German Research Foundation) within SPP 2410, project number 525935467.

\def\cprime{$'$} \def\ocirc#1{\ifmmode\setbox0=\hbox{$#1$}\dimen0=\ht0
	\advance\dimen0 by1pt\rlap{\hbox to\wd0{\hss\raise\dimen0
			\hbox{\hskip.2em$\scriptscriptstyle\circ$}\hss}}#1\else {\accent"17 #1}\fi}


\end{document}